\documentclass[a4paper,11pt,twoside,reqno]{amsart}  

\usepackage{soul} 
\setstcolor{blue}

\DeclareSymbolFontAlphabet{\mathcal}{symbols}
\usepackage[colorlinks=red,linkcolor=blue]{hyperref}
\usepackage[all]{xy}
\usepackage{amsmath}
\usepackage[psamsfonts]{amssymb}
\usepackage[psamsfonts]{eucal}
\usepackage{amsthm}
\usepackage[scaled]{helvet}
\usepackage[a4paper]{geometry}
\usepackage{hyperref}
\usepackage{cleveref}
\usepackage{mathrsfs}
\usepackage{color}

\usepackage[framemethod=tikz]{mdframed}
 \mdfdefinestyle{MyFrame}{%
    linecolor=blue,
    leftline=false,
    topline=false,
    bottomline=false,
    innertopmargin=0.1\baselineskip,
    innerbottommargin=0.5\baselineskip,
    innerrightmargin=10pt,
    innerleftmargin=0pt   
    }
\usepackage{enumerate}
\theoremstyle{plain}

\newtheorem{proposition}{Proposition}[section]
\newtheorem{theoremb}[proposition]{Theorem}
\newtheorem{theoremv}{Main Theorem}

\newtheorem{lemma}[proposition]{Lemma}
\newtheorem{corollary}[proposition]{Corollary}

\theoremstyle{definition}
\newtheorem{definition}[proposition]{Definition}
\newtheorem{example}[proposition]{Example}

\theoremstyle{remark}
\newtheorem{remark}[proposition]{Remark}

\newcommand{\secref}[1]{Section~\ref{#1}}

\newcommand{\thmref}[1]{Theorem~\ref{#1}}
\newcommand{\propref}[1]{Proposition~\ref{#1}}
\newcommand{\lemref}[1]{Lemma~\ref{#1}}
\newcommand{\corref}[1]{Corollary~\ref{#1}}

\newcommand{\exemref}[1]{Example~\ref{#1}}

\newcommand{\defref}[1]{Definition~\ref{#1}}

\makeatletter
\renewcommand\l@subsection{\@tocline{2}{0pt}{2pc}{5pc}{}}
\renewcommand\l@subsubsection{\@tocline{3}{0pt}{4pc}{10pc}{}}
\makeatother

\def\R{{\mathbb R}}

\def\HH{\mathbb{H}}

\def\R{\mathbb{R}}

\def\ker{{\rm Ker\,}}

\def\Sp{{\rm Sp}}

\def\id{{\rm id}}

\newcommand{\Hprod}[1]{\langle #1\rangle}
\newcommand{\diag}{\mathrm{diag}}		
\newcommand{\Tr}{\mathrm{Tr}} 		
\newcommand{\grad}{\mathrm{grad}\,}	

\newcommand{\block}[2]{\addtolength{\arraycolsep}{-3pt}\begin{bmatrix}#1\cr#2\end{bmatrix}}
\makeatletter
\renewcommand*\env@matrix[1][*\c@MaxMatrixCols c]{%
  \hskip 0pt
  \let\@ifnextchar\new@ifnextchar
  \array{#1}}
\makeatother

\newcommand{\proj}{{\rm proj}}
\newcommand{\Hess}{\mathrm{H}}
\newcommand{\PP}{\mathcal{P}}
\newcommand{\A}{\mathfrak{A}}
\newcommand{\Gr}{\mathrm{Gr}}

\begin{document}


\title[Non-linear Morse-Bott functions]{Non-linear Morse-Bott functions\\ on quaternionic Stiefel manifolds}

\author[E. Mac\'{\i}as-Virg\'os]{Enrique Mac\'{\i}as-Virg\'os}
\address{Institute of Mathematics\\
Department of Geometry and Topology\\
Universidade de Santiago de Compostela\\
15782 Spain}
\email{quique.macias@usc.es}

\author[M.-J. Pereira-S\'aez]{Mar\'{\i}a~Jos\'e~Pereira-S\'aez}
\address{ Diff. Geom. and appl. Research Group, 
		Universidade da Coru\~na \\
		15071
		Spain.}
\email{maria.jose.pereira@udc.es}

\author[D. Tanr\'e]{Daniel Tanr\'e}
\address{D\'epartement de Math\'ematiques \\
Universit\'e de Lille \\
59655 Villeneuve d'Ascq Cedex, France}
\email{Daniel.Tanre@univ-lille.fr}

\date{\today}

\begin{abstract} 
In the Stiefel manifold $X_{n,k}$, we replace Frankel linear height function 
%
by  a quadratic one.
%
We prove this is still a Morse-Bott function, whose  structure of
 critical levels presents a dichotomy according to the sign  of $n-2k$.
The critical submanifolds are no longer
Grassmannians but total spaces of fibrations of basis a product of two Grassmannians. We explicitly integrate the gradient flow.
\end{abstract}

\thanks{The three authors are  partially supported by the MINECO and 
FEDER research project MTM2016-78647-P.  The first author was partially supported by Xunta de Galicia ED431C 2019/10 with FEDER funds}



\maketitle


\tableofcontents

\newpage
\section*{Introduction}

Morse-Bott functions are a very effective tool for studying compact manifolds.
The pioneer work of T. Frankel (\cite{Frankel1963}) on Lie groups and Stiefel manifolds 
uses height functions defined, on the unitary group for instance, by the real part of the trace,
$A\mapsto \Re\Tr A$.
The associated critical submanifolds are Grassmannians and from this fact
H. Miller (\cite{MR816522}) proves that Stiefel manifolds admit a stable decomposition as Thom spaces of bundles over
these Grassmannians, see also \cite{MR1853464}.
These critical submanifolds have also been the origin of the work of H. Kadzisa and M. Mimura (\cite{KadzMimu2011})
for the construction of cone-decompositions of Stiefel manifolds and their application to Lusternik-Schnirelmann category. Some recent papers took up again the work of Frankel, by using linear and quadratic Morse-Bott functions on orthogonal groups and Stiefel manifolds,
as in \cite{MR2642452,MR3982635}.

\medskip
 In \cite{MR3621032},  height functions
 with respect to the hyperplan orthogonal to a given matrix $\omega$ are considered on quaternionic Stiefel manifolds. 
 All  those functions are Morse-Bott with Grassmannians as critical submanifolds and,
among them, the Morse functions are  characterized  according to  $\omega$.
Some consequences on the
Lusternik-Schnirelmann category are given.
In this work, we replace height functions by a non-linear function and prove that  the situation is different:
for  instance,  instead of  Grassmannians,
the critical submanifolds are  associated fibrations to the principal bundles of a product of two Stiefel manifolds, 
with basis a product of two Grassmannians, see \propref{FIBERBUNDLE}. 

Finally, we are able to explicitly integrate the gradient flow, for arbitrary initial data.

\medskip
Let us specify the notations used.
Let $\HH^n$ be the quaternionic $n$-space endowed with
the structure of a right $\HH$-vector space and the hermitian product $\Hprod{u,v}=u^*v$.
For $0\leq k\leq n$, let $X_{n,k}$ be the Stiefel manifold of linear maps $ \HH^k \to \HH^n$  preserving the  Hermitian product. Such map is identified with a matrix $x\in \HH^{n\times k}$  
represented by two blocks,
$$x=\block{T}{P},$$
where $T,P$ are quaternionic matrices of order $(n-k)\times k$ and $k\times k$, respectively. The condition
$x^*x=I_{k}$ becomes
$T^*T+P^*P=I_k$.
Let $\Sp(n)$ be the Lie group of $n\times n$ matrices $A$ such that $A^*A=I_n$. The linear left action of $\Sp(n)$ on $X_{n,k}$ is transitive and the isotropy group of $x_0=\block{0}{I_{k}}$ is isomorphic to $\Sp(n-k)$, so $X_{n,k}$ is diffeomorphic to $\Sp(n)/\Sp(n-k)$.

\medskip
If  $x=\block{T}{P}\in X_{n,k}$, we define the real number $h(x)$  as the trace of $P^*P$,
\begin{equation}\label{equa:h}
h(x)=\Tr{(P^*P)}.
\end{equation}
The purpose of this work is the study of the function $h\colon X_{n,k}\to \R$. 
In what follows we always assume  $k<n$ since we have $X_{n,n}=\Sp(n)$ where
the function $h$ is constant.

\begin{theoremv}
Let $0\leq k<n$ and $X_{n,k}=\Sp(n)/\Sp(n-k)$ be the  quaternionic Stiefel manifold.
The following properties are satisfied.
\begin{enumerate}[1)]
\item The function $h\colon x\mapsto \Tr{(P^*P)}$ is Morse-Bott.
\item The critical levels of $h$ are
\begin{itemize}
\item $q=0,\dots, k$, if $n\geq 2k$, 
\item $q=2k-n,\dots,k$, if $n\leq 2k$.
\end{itemize}
\item The critical submanifold $\Sigma_{q}$ is 
the total space of a principal fibration with fibre $\Sp(k)$ and 
base space the product of Grassmannians $\Gr_{n-k,p}\times \Gr_{k,k-p}$.
The index of $\Sigma_{q}$ is $4(n-2k+q)q$, with $q=k-p$.
\end{enumerate}
\end{theoremv}

Let us observe that the dichotomy between $n\geq 2k$ and $n<2k$, appearing in the previous statement, 
 is also reflected in Nishimoto's work (\cite{Nishimoto2007}) 
 on the Lusternik-Schnirelmann category of quaternionic Stiefel manifolds $X_{n,k}$.
 
 \medskip
 The proof of the Main Theorem occupies most of the rest of this work. 
 In \secref{sec:gradient}, we compute the gradient of $h$ and prove Assertion 2)  of the Main Theorem. 
 In \secref{sec:hessian}, we give explicit expressions of the Hessian
whereas the determination of its eigenvalues and eigenvectors is done in \secref{sec:eigenalgo}.
 In \secref{sec:criticalsubmanifolds}, we show that the group
 $K_{n,k}=\Sp(n-k)\times \Sp(k) \times \Sp(k)$
 acts transitively on the critical submanifold~$\Sigma_{q}$. 
 For this action, we  also determine the isotropy subgroup 
 as
$$L_{n,k,q}=\Sp(k-q)\times \Sp(n+q-2k)\times \Sp(k-q)\times \Sp(q).$$
The inclusion  $L_{n,k,q}\hookrightarrow K_{n,q}$, described in \corref{ISOTROPY},
brings up diagonal terms on certain factors 
 and thus the writing of $\Sigma_{q}$
 as a homogeneous space is not direct. Still, we express it as announced in the Assertion 3) 
 and detail some particular cases.
Assertion 1) is also a consequence of this determination, see \corref{cor:todobien}.
 
 Finally, 	in Section \ref{integration:flow}, we give an explicit description of the gradient flow of $h$.Ó

\section{Critical points}\label{sec:gradient}
Let $h\colon X_{n,k}\to\R$ be the function defined  
by $  h(x)= \Tr{(P^*P)}$
for $x=\block{T}{P}$. 

\subsection{Gradient} 
In this paragraph, we explicit the value of the gradient of $h$.

\begin{proposition}\label{prop:gradient}
The gradient of $h$ at the point $x=\block{T}{P}$ is given by
$$\grad h_x =-2\block{TP^*P}{(PP^*-I_{k})P}.$$
\end{proposition}

We proceed as usual by immersing the manifold $X_{n,k}$  in 
$\HH^{n\times k}=\R^{4n\times 4k}$, where the Euclidean metric is given by
$\vert x \vert^2 =\Tr(x^*x)$. The gradient of $h$ is obtained from
the gradient of an extension  $\phi\colon \R^{4n\times 4k}\to\R$
of $h$ that we project on the tangent space $T_{x} X_{n,k}$ of $X_{n,k}$:
$$\grad h_x=\proj_x(\grad \phi_x),$$
where we denote by $\proj_x$ the projection onto the tangent space.
Let us begin by some lemmas in this direction.

\begin{lemma}\label{lem:GRADEXT}
The function
 $\phi\colon \R^{4n\times 4k} \to \R$, defined by $\phi(x)=\Tr(P^*P)$ for $x=\block{T}{P}$, extends $h$ to the whole Euclidean space.
If $x=\block{T}{P}$ then $\grad \phi_x=2\block{0}{P}$.
\end{lemma}

\begin{proof}
The differential $d\phi_x \colon T_x\R^{4n\times 4k}\to T_{\phi(x)}\R$ is given by
$d\phi_x(u)=\lim_{t\to 0}\frac{1}{t}(\phi(x+tu)-\phi(x))$.
If $u=\block{\tau}{\pi}$ with $\tau\in \HH^{n-k\times k}$ and $\pi\in\HH^{k\times k}$, then we have 
$d\phi_x(u)=2\Re \Tr(P^*\pi)$. 

By definition, if $v=\grad\phi_x=\block{C}{D}$, then we have 
$d\phi_x(u)=\langle v,u\rangle =\Re\Tr(v^*u)$, for all $u\in T_x\R^{4n\times 4k}$, that is,
$$2\Re\Tr(P^*\pi)=\Re\Tr(C^*\tau+D^*\pi),$$ for all $\tau$ and $\pi$.
This implies $C=0$ and $D=2P$.
\end{proof}

For the next step, we need to determine the tangent space  $T_{x} X_{n,k}$. 
Let us begin with $x_0=\block{0}{I_k}$. The tangent space in $x_{0}$ is formed by the vectors
$v=\block{X}{Y}$, $X\in\HH^{n-k\times k}$, $Y\in \HH^{k\times k}$,  with $Y+Y^*=0$. 
(This corresponds to the equation $x_0^*v+v^*x_0=0$.)
Any other point of the Stiefel manifold can be written as $x=Ax_0$, with $A\in\Sp(n)$, so the tangent space 
 is given by $T_{x}X_{n,k}=A \,T_{x_0}X_{n,k}$.

\begin{lemma}\label{lem:NORMPROJECTION}
If $x\in X_{n,k}$, the projection of a vector $u\in\R^{4n\times 4k}$ onto the {\em normal} vector space $\nu_x$  is
$$\proj_x^\perp u =\frac{1}{2}{x(x^*u+u^*x)}.$$
\end{lemma}
\begin{proof}
We begin with $x_0=\block{0}{I_k}$. If $u=\block{A}{B}$, then its tangent part is $u^\top=\block{A}{(1/2)(B-B^*)}$ and its normal part is 
$$u^\perp=\block{0}{(1/2)(B+B^*)}=\frac{1}{2}x_0(x^*_0u+u^*x_0).$$ This can be checked by proving that they are orthogonal for the Hermitian product $v^*w$ and that the first one lies on the tangent space described above.
 Now,  in general, since $x=Ax_0$ and the product with $A$ is an isometry, we have:
 $$\proj_x^\perp u =A\,\proj^\perp_{x_0}(A^*u)=A\,(\frac{1}{2}x_0(x^*_0A^*u+u^*Ax_0))=\frac{1}{2}{x(x^*u+u^*x)}. 
 \qedhere$$
\end{proof}

\begin{proof}[Proof of \propref{prop:gradient}]
From Lemmas \ref{lem:GRADEXT} and \ref{lem:NORMPROJECTION}, the normal projection of $u=\grad \phi_x$ is
$$\proj^\perp_x u=\frac{1}{2}{x\left(x^*\block{0}{2P}+\block{0}{2P}^*x\right)}=\block{T}{P}(2P^*P),$$
so  $$\grad h_x=\proj_x u=u-\proj_x^\perp u=\block{0}{2P}-\block{2TP^*P}{2PP^*P},$$
and the result follows.
\end{proof}

\subsection{Critical points and levels}
The previous determination of the gradient gives a characterization of the critical points and values.

\begin{proposition}\label{prop:points}
 \mbox{}
\begin{enumerate}[1)]
\item The point $x=\block{T}{P}$ is a critical point of $h$ if, and only if, 
$TP^*=0$.
\item If  $n\geq 2k$,  the function $h$ has  $k+1$ critical values $q=0,\dots, k$.
\item If  $n\leq 2k$, then $h$ has  $n-k+1$ critical values, $q=2k-n,\dots,k$.
\end{enumerate}
\end{proposition}

\begin{proof}
1) The condition $\grad h_x=0$ implies that $P=PP^*P$, so $P^*P=(P^*P)^2$. But since $P\in\HH^{k\times k}$ is a square matrix, it follows that $P^*P$ is a Hermitian matrix $S$, which is semidefinite positive since for any
$u\in\HH^k$, $u\neq 0$, we have
$$u^*Su= \langle u,Su\rangle=\langle u,P^*Pu\rangle=\langle Pu,Pu\rangle=\vert Pu\vert^2\geq 0.$$
The eigenvalues of $S$ are thus non-negative real numbers  $s_1,\dots,s_k$, verifying $s_i^2=s_i$, so $s_i=0,1$.
We  write a singular value decomposition (henceforth SVD) of $P$ (\cite{MR2990115}) as
$$P=a\begin{bmatrix}0_p&0\cr 0 &I_q\end{bmatrix}b^*, \quad a,b\in\Sp(k), \quad p+q=k,$$
and the corresponding diagonalization
$$S=P^*P=b\begin{bmatrix}0_p&0\cr 0 &I_q\end{bmatrix}b^*.$$
The condition $\grad h_x=0$ in \propref{prop:gradient} also implies $TP^*P=0$, so
$$Tb\begin{bmatrix}0_p&0\cr 0 &I_q\end{bmatrix}b^*=0$$
and, using  $b^*b=I_k$, 
$$TP^*=Tb\begin{bmatrix}0_p&0\cr 0 &I_q\end{bmatrix}b^*ba^*=0.$$
For the reciprocal, it is clear that the two conditions  $TP^*=0$ and $PP^*P=P$ imply $\grad h_x=0$. But the second 
condition is implied by the first one, since
$TP^*=0$ implies $(I_{k}-P^*P)P^*=T^*TP^*=0$ and $P^*PP^*=P^*$.

\medskip
2) Let $x$ be a critical point and suppose $n-k\geq k$. From the relative SVD developed in \cite{MR3990067},
there exists a SVD of $T$ as
$$m\block{D}{0}b^*, \quad D=\diag[t_1,\dots ,t_{k}].$$
So $T^*=b\begin{bmatrix}D& 0\cr\end{bmatrix}m^*$,
$T^*T=bD^2b^*$  
and the condition $T^*T+P^*P=I_{k}$ implies 
$$D^2+\begin{bmatrix}0_p&0\cr 0 &I_q\end{bmatrix}=I_k$$
then $D=\begin{bmatrix}I_p&0\cr 0 &0_q\end{bmatrix}$.
From $q=k-p$ and $k\leq n-k$, we deduce $0\leq q\leq k$.

\medskip
3) Let $x$ be a critical point and suppose $n-k< k$. We write
$$T=m\begin{bmatrix}D&0\cr\end{bmatrix}b^*, \quad D=\diag[t_1,\dots ,t_{n-k}].$$
The same argument gives $D=\begin{bmatrix}I_p&0\cr 0 &0_{n-k-p}\end{bmatrix}$ and $0\leq p\leq n-k$.
With $q=k-p$, this gives $2k-n\leq q\leq k$.
\end{proof}

Let us detail some easy examples of critical submanifolds.

\begin{example} If $k=1$, then $n>1$ hence $n-k\geq 1=k$ and the  function $h$ has two critical levels $q=0,1$.
In fact,  $X_{n,1}=\Sp(n)/\Sp(n-1)$ is the sphere $S^{4n-1}$ and $P\in \HH^{1\times 1}$ is a quaternion 
verifying  $P^*P=\vert P\vert^2\leq 1$. 
The function $h$, defined by $h(x)=\vert P\vert^2$, has~1 as maximum and 0 as minimum  values.
\end{example}

\begin{example}[\bf Maximum level]\label{exa:maxlevel}
The maximum  level $q=k$  implies  $P^*P=I_k$ and  $P=ab^*\in \Sp(k)$,
hence $T^*T=0$ and  $T=0$.
So, the critical level $\Sigma_{k}$ is the space of matrices $\block{0}{P}$  
diffeomorphic to the symplectic group $\Sp(k)$ which can also be viewed as the Stiefel manifold $X_{k,k}$. 
\end{example}

\begin{example}[\bf Minimum level]\label{exa:minlevel}
Suppose  $n\geq 2k$. The minimum level corresponds to $q=0$, thus $p=k$ and $P=0$. We deduce 
$T^*T=I_k-P^*P=I_k$, which means that $x=\block{T}{0}\in X_{n,n-k}$. Thus $\Sigma_{0}=X_{n-k,k}$.

The second case, $n-k< k$, is more involved and is considered in \exemref{EXAMPLEINVOLVED}. 
\end{example}

\section{Hessian}\label{sec:hessian}

The tangent vector field $\grad h$ 
when differentiated at a critical point $x\in X_{n,k}$ gives rise to a linear map
\begin{equation}\label{equa:hess}
\Hess h_x\colon v\in   T_xX_{n,k} \mapsto \nabla_v\grad h\in T_xX_{n,k},
\end{equation}
called the Hessian of $H$. We determine it in this section but, before, let us recall some notations.

\medskip
\paragraph{\bf Notations}
Let $x_{0}=\block{0}{I_{k}}$. For any $x=\block{T}{P}\in X_{n,k}$, there exists an element
$A=\begin{bmatrix}\alpha&T\cr \beta &P\cr\end{bmatrix}
\in\Sp(n)$
such that $x=A x_{0}$. Moreover, as 
$T_{x}X_{n,k}=A\,T_{x_{0}}X_{n,k}$, any vector $v\in T_{x}X_{n,k}$
can be written as $v=A\,\block{X}{Y}$ with $Y+Y^*=0$. 
We  keep these notations all along this section.

\medskip
Now comes the main result of this section.

\begin{theoremb}\label{HESSIANFORMULA}
The Hessian of the map $h(x)=\Tr(P^*P)$, at the critical point $x$ and the vector 
$v\in T_{x}X_{n,k}$, is given by 
\begin{eqnarray}
\Hess h_x(v)
&=&
-2(vx^*+xv^*)x_0P -2\block{0}{PP^*-I}x_0^*v \nonumber\\
&=&
-2(vx^*+xv^*)x_0P+2\block{0}{\beta X}. \label{LONGER}
\end{eqnarray}
\end{theoremb}

\begin{remark}
The first expression shows that the Hessian only depends on $
x$ and $v$. In this form, it is not obvious at all that this vector belongs to the tangent space $T_xX_{n,k}$. 
Replacing
$x$, $x_{0}$ and $v$  by their value in \eqref{LONGER} shows that
\begin{equation}\label{COMPACT}
\Hess h_x(v)=
-2A\block{XP^*P-\beta^*\beta X}{X^*\beta^*P-P^*\beta X}
\end{equation} 
Since the matrix $X^*\beta^*P-P^*\beta X$ is skew-Hermitian, we can now notice that
$\Hess h_x(v)\in A\cdot T_{x_0}X_{n,k}=T_{x}X_{n,k}$. 
\end{remark}

To  prove  this theorem,  let us  go back to the beginning. In \eqref{equa:hess},
$\nabla$ is the connection on the Riemannian submanifold $X_{n,k}\subset \R^N$, with $N={4n\times 4k}$, so if $\widetilde{\grad h}\colon \R^N\to \R^N$ is a vector field extending $\grad h$ to all the Euclidean space, it is 
(see \cite{MR909697})
$$\nabla_v\grad h=\proj_x D_{x}(\widetilde{\grad h})(v).$$ 
We can  denote it
$$\proj_x D_v(\widetilde{\grad h})$$
since the connection in  $\R^N$ is the usual derivative. Hence, we have
\begin{equation}\label{HESSFORMULA}
\Hess h_x(v)=\proj_x \frac{d}{dt}_{\vert t=0}\widetilde{\grad h}_{x+tv}.
\end{equation}

This situation is covered by \cite{MR3126065} which contains  an explicit formula for 
the Hessian of any map $h\colon M \to \R$,
 which is the restriction to a Riemannian submanifold $M\subset \R^N$  
 of a function $\phi\colon \R^N \to \R$ defined in the Euclidean space. 
 We follow this pattern for $M=X_{n,k}$ and $h$, $\phi$ the functions defined in \secref{sec:gradient}.
 Recall that the authors of \cite{MR3126065} introduce
 \begin{itemize}
 \item the orthogonal projectors on the tangent and the normal spaces, 
 $\PP_{x}\colon \R^N\to T_{x}X_{n,k}$ and $\PP_{x}^\perp=\id-\PP_{x}\colon \R^N\to \nu_{x}$,
 \item the Weingarten map, $\A_{x}\colon T_x X_{n,k}\times \nu_x \to T_xX_{n,k}$, of $X_{n,k}$ at $x$, 
 which is  given by
$$\A_{x}(v,w)=-\PP_xD_vW,$$
where $W$ is any normal vector field  extending locally the vector $w$,
 \end{itemize}
and  they prove the two following results.

\begin{proposition}\label{LEMAABS}\cite[Theorem 1]{MR3126065} 
For any $x\in X_{n,k}$, $u\in\R^N$ and $v\in T_{x}X_{n,k}$, the Weingarten map verifies
\begin{equation}\label{equa:wein}
\A_{x}(v,\PP_x^\perp u)=\PP_xD_v\PP_{x}\PP_{x}^\perp u.
\end{equation}
\end{proposition}

\begin{proposition}\cite[Formula (10)]{MR3126065}\label{prop:HESSABIL}
The Hessian of the map $h$ whose extension is $\phi$ is given by
\begin{equation}\label{HESSABIL}
\Hess h_x (v)=\nabla_v\grad h=\PP_x\Hess \phi_x(v)+\A(v,\PP_x^\perp\grad \phi_x).
\end{equation}
\end{proposition}

We develop some lemmas for the computation of the different expressions in the  formulae 
\eqref{equa:wein} and \eqref{HESSABIL}.
Let us begin with the orthogonal projection.
We already know (\lemref{lem:NORMPROJECTION}) that the normal projection, $\proj^\perp _x\colon \R^N \to \nu_x$, 
is defined, for each point $x\in X_{n,k}$ as
$$\proj^\perp_x(u)=\frac{1}{2}x(x^*u+u^*x).$$
Thus we extend the operator $\proj_x$  to the whole space by
$$ \PP_x(u)=u-\frac{1}{2}x(x^*u+u^*x), \quad \text{ for any }x,u\in\R^N.$$
\begin{lemma}For each $x\in \R^N$, 
the $\R$-linear map $\PP_x\colon \R^N\to \R^N$ is a projector, that is, it  verifies $\PP_x\circ\PP_x=\PP_x$.
\end{lemma}

\begin{proof} From
$$x^*\PP_x(u)=x^*u-\frac{1}{2}x^*x(x^*u+u^*x)=\frac{1}{2}(x^*u-u^*x),$$
we deduce
$$(x^*\PP_x(u))^*=-x^*\PP_x(u)$$
and
$$\PP_x\PP_x(u)=\PP_x(u)-\frac{1}{2}x(x^*\PP_x(u)+\PP_x(u)^*x)=\PP_x(u).\qedhere$$
\end{proof}

Then we consider the Hessian of the extension $\phi$. 

\begin{lemma}
The Hessian of the map $\phi\colon \R^{4n\times 4k}\to \R$ is
$$\Hess \phi_x (v) =2\block{0}{\beta X + PY}.$$ 
\end{lemma}

Notice that $\beta X +PY=x_0^*v$ only depends on $v$.

\begin{proof}
Recall from \lemref{lem:GRADEXT} that $\grad \phi_x=\block{0}{2P}$. Thus, we have
$$\Hess \phi_x (v) =\frac{d}{dt}_{\vert t=0}\grad \phi_{x+tv}=2\frac{d}{dt}_{\vert t=0}\block{0}{P+t(\beta X +PY)}=2\block{0}{\beta X + PY}.\qedhere$$
\end{proof}

\begin{lemma}\label{NORMALHESS} 
The normal projection of the Hessian of $\phi$ is 
$$\proj^\perp_x\Hess\phi_x(v)= x\left(P^*(\beta X + PY)+(\beta X + PY)^*P\right).$$
\end{lemma}

\begin{proof}
This follows from a direct computation
\begin{eqnarray*}
\proj_x^\perp\Hess \phi_x (v)
&=&
\frac{1}{2}x(x^*\Hess \phi_x (v)+\Hess \phi_x (v)^*x)\\
&=&
x\left(\block{T}{P}^*\block{0}{\beta X + P Y}+\block{0}{\beta X + P Y}^*\;\block{\;\;T}{\;\;P}\right).\qedhere
\end{eqnarray*}
\end{proof}

Finally, we specify the Weingarten map.

\begin{lemma} \label{WEING}
\mbox{}
\begin{enumerate}[1)]
\item The  Weingarten map  of the Stiefel manifold $X_{n,k}$
is given by
\begin{equation}\label{SHAPE}
\A_x(v, w ) = -vx^*w -\frac{1}{2} x(v^*w + w^*v),\quad v\in T_xX_{n,k}, \;w\in \nu_x.
\end{equation}
\item In particular, if $x$ is a critical point and
 $v=A\block{X}{Y}\in T_xX_{n,k}$,  we have
  $$\A_x(v,\proj_x^\perp\grad \phi_x)=
-2vP^*P-x\left( P^*(\beta X + P Y)+(\beta X + P Y)^*P\right).$$
\end{enumerate}
\end{lemma}

\begin{proof}
1) This formula is established in \cite[Section 4.1]{MR3126065} for the real Stiefel manifolds. 
The same proof works, word for word, in the quaternionic case.

\medskip
2) Notice first that, in a critical point $x$ for $h$, we have $\proj_x\grad \phi_x=\grad h_x =0$, thus
$\proj_x^\perp\grad \phi_x=\grad \phi_x$.
Then, from formula \eqref{SHAPE}, we deduce
\begin{eqnarray*}
\A_x(v,\proj_x^\perp\grad \phi_x)
&=&
\A_x\left(v,\block{0}{2P}\right)\\
&=&
-  vx^*\block{0}{2P}-\frac{1}{2}x\left(v^*\block{0}{2P}+\block{0}{\;2P}^*v\right)\\
&=&
-2vP^*P-\frac{1}{2}x\left(\block{X}{Y}^*\block{2\beta^*P}{2P^*P}+
\begin{bmatrix}2P^*\beta&\ 2P^*P\cr\end{bmatrix}\block{\;\;X}{\;\;Y}\right)\\
&=&
-2vP^*P- x\left(\block{X}{Y}^*\block{\beta^*}{P^*}P+P^*\begin{bmatrix}\beta&P\cr\end{bmatrix}\block{\;\;X}{\;\;Y}\right)\\
&=&
-2vP^*P-x\left( (\beta X + P Y)^*P+P^*(\beta X + P Y)\right).\qedhere
\end{eqnarray*}
\end{proof}

From these lemmas, we can now prove the formula given for the Hessian.

\begin{proof}[Proof of \thmref{HESSIANFORMULA}]
Recall $x_0^*v=\beta X + PY$.
With \lemref{NORMALHESS}, we have 
$$\proj_x\Hess\phi_x(v)=2\block{0}{x_0^*v}- x(P^*x_0^*v+v^*x_0P).$$
From \lemref{WEING}, we compute the Weingarten map: 
$$\A_x(v,\proj_x^\perp\grad \phi_x)=
-2vP^*P-x\big( P^*x_0^*v+v^*x_0 P\big).$$
From \propref{prop:HESSABIL}, we get
\begin{eqnarray*}
\Hess h_x (v)
&=&
2\block{0}{x_0^*v}-2vP^*P-2x(P^*x_0^*v+v^*x_0P)\\
&=&
2\block{0}{x_0^*v}-2vx^*x_0P-2xP^*x_0^*v-2xv^*x_0P\\
&=&
2\block{0}{x_0^*v}-2vx^*x_0P-2\block{T}{P}P^*x_0^*v-2xv^*x_0P\\
&=&
-2\block{0}{PP^*-I}x_0^*v-2(vx^*+ xv^*)x_0P,
\end{eqnarray*}
since $TP^*=0$ for a critical point. 
\end{proof}
\section{Eigenvalues and eigenvectors}\label{sec:eigenalgo}

In this section, we compute the eigenvalues of the Hessian,
by solving the equation $\Hess h_x(v)=v\lambda$, where $x\in\Sigma_{q}$, $v\in T_xX_{n,k}$ and $\lambda\in\HH$.
Recall that the critical submanifold is nondegenerate if the kernel of the Hessian coincides with the tangent space
at each point. Also, the index of $\Sigma_{q}$
 is the dimension of the largest subspace on which the Hessian is negative definite.

\begin{example}\label{exa:elmassencillo}
Let us begin with the simplest case: $x_0=\block{0}{I}$. Since it verifies the condition $TP^*=0$, this is a
critical point. In fact, we have $h(x_0)=k$, so it is a maximum.
For the value of the Hessian, we apply \eqref{COMPACT}, with $A=I_{n}$, $P=I_{k}$ and $\beta=0$ and get
the equation giving the eigenvalues:
$$\Hess h_{x_0}(v)=-2\block{X}{0} =\block{X\lambda}{Y\lambda}.$$
Therefore, the eigenvalues are $\lambda=0$ with multiplicity $k$, and $\lambda=-2$
with multiplicity $n-k$ equal to the index, showing that the transverse directions are all going down. 
Below, we develop this structure in the general case.
\end{example}

As it will soon appear, it is sufficient to  determine the eigenvalues in some particular points,
as $x_{0}$ for the case $q=k$.  We construct now such particular point in each singular submanifold.
Let $p,q\geq 0$, $p+q=k$ and $p+q'=n-k$.

\begin{definition}\label{def:mascomplicadoPero}
The \emph{notable point} $x_{0}^q$ of $\Sigma_{q}$
is defined by
$$x_{0}^q=\block{T_0}{P_0}
\quad \text{with}
\quad
P_0=\begin{bmatrix}0_p&0\cr 0 &I_q\end{bmatrix}
\quad
\text{and}
\quad
T_0=\begin{bmatrix}I_{p}&0\cr 0& 0_{q'\times q}\cr\end{bmatrix}.
$$
\end{definition}

Indeed, an easy computation gives the equalities 
$T_{0}^*T_{0}+P_{0}^*P_{0}=I_{k}$ and
$T_{0}P_{0}^*=0$
which justify the assertion $x_{0}^q\in \Sigma_{q}$.
(The previous point $x_{0}$ of \exemref{exa:elmassencillo} is $x_{0}^k$.)

\medskip
Let us return to the general case. Using \eqref{COMPACT},  the eigenvalues are the solutions of the system
\begin{equation}\label{equa:eigenvalues}
\left\{
\begin{array}{lcl}
-2(XP^*P-\beta^*\beta X)
&=&
X\lambda\\
-2(X^*\beta^*P -P^*\beta X)
& =&Y \lambda.
\end{array}
\right.
\end{equation}
Instead of solving it in general, we  show its invariance  by a transitive action, which reduces the resolution to
the case of the notable  points.

\begin{theoremb}
\mbox{}
\begin{enumerate}[1)]
\item The group $K_{n,k}=\Sp(n-k)\times \Sp(k)\times \Sp(k)$ acts transitively on the left on each critical level $\Sigma_q$ as
$$(m,a,b)\cdot x=(m,a,b)\cdot \block{T}{P}=\block{mTb^*}{aPb^*}=
\begin{bmatrix}m&\;\;0\cr 0&\;\;a\cr\end{bmatrix}\block{\;\;T}{\;\;P}b^*.$$
\item
The Hessian is invariant by the action, that is if $x=g\cdot x_0^q\in\Sigma_q$ and $v=g\cdot v_0\in T_xX_{n,k}$,
with $g\in K_{n,k}$  and $v_0\in T_{x_0^q}X_{n,k}$, then
$$\Hess h_x(v)=g\cdot \Hess h_{x_0^q}(v_0).$$
\end{enumerate}
\end{theoremb}

\begin{proof} 
1) Let $x\in\Sigma_{q}$.
Set $T'=mTb^*$ and $P'=aPb^*$. We verify
$T'(P')^*=mTb^*bP^*a=m(TP^*)a=0$, thus,
from  \propref{prop:points}, $(m,a,b)\cdot x$ is a critical point. We compute now its image by $h$ and 
get directly from the definition,
$$h((m,a,b)\cdot x)=\Tr((P')^*P')=\Tr(bP^*a^*aPb^*)=\Tr(bP^*Pb^*)=\Tr(P^*P)=q.$$
Thus $(m,a,b)\cdot x\in \Sigma_{q}$.
Let $x_{0}^q$ be the notable point of $\Sigma_{q}$. From \cite[Theorem 3.1]{MR3990067},
there  exist  $a,b\in \Sp(k)$ and $m\in\Sp(n-k)$ such that
$P=aP_0b^*$ and $T=mT_0b^*$. This implies the transitivity of the action.

\medskip
2) If $x=Ax_0$, $x=g\cdot x_0^q$, and $x_0^q=Bx_0$, 
 where $A,B\in\Sp(n)$, $x_0=x_0^k=\block{0}{I_k}$
 and $T_{0}$ and $P_{0}$ as in \defref{def:mascomplicadoPero}.
 We first determine $A$ and $B$. Let us observe,
\begin{equation}\label{equa:AandB}
Ax_0=x=g\cdot x_0^q
=
\begin{bmatrix}m&0\cr 0&a\cr\end{bmatrix}\block{T_0}{P_0}b^*
=
\begin{bmatrix}m&0\cr 0&a\cr\end{bmatrix}Bx_0b^*=\begin{bmatrix}m&0\cr 0&a\cr\end{bmatrix}B\begin{bmatrix}I_p&0\cr 0&b^*\cr\end{bmatrix}x_0.
\end{equation}
Thus, we can take 
$$A= \begin{bmatrix}m&0\cr 0&a\cr\end{bmatrix}B\begin{bmatrix}I_p&0\cr 0&b^*\cr\end{bmatrix}$$ 
 so
$B=\begin{bmatrix}m^*&0\cr 0&a^*\cr\end{bmatrix}A\begin{bmatrix}I_p&0\cr 0&b\cr\end{bmatrix}$.
If we denote 
$A= \begin{bmatrix}\alpha&T\cr \beta &P\cr\end{bmatrix}
$
and $B=\begin{bmatrix}\alpha_{0}&T_0\cr \beta_{0} &P_0\cr\end{bmatrix}
$, the equality \eqref{equa:AandB} implies $\beta_0=a^*\beta$.

Now, let $v=A\block{X}{Y}=g\cdot v_0=g\cdot B\block{X_0}{Y_0}$, we have
$$A\block{X}{Y}
=
\begin{bmatrix}m&\;0\cr \;0&a\cr\end{bmatrix}B\block{X_0}{Y_0}b^*
=
A\begin{bmatrix}I_p&\;\;0\cr 0&\;\;b\cr\end{bmatrix}\block{\;\;X_0}{\;\;Y_0}b^*
$$
which gives
$
X_0=Xb$ and $ Y_0=b^*Yb$. 
By replacing $X=X_{0}b^*$, $P=aP_{0}b^*$ and $\beta=a\beta_{0}$ by their values, we get
$$\left\{
\begin{array}{ccl}
XP^*P-\beta^*\beta X
&=&
(X_{0}P_{0^*}P_{0}-\beta_{0}^*\beta_{0}X_{0})b^*,\\
X^*\beta^*P-P^*\beta X
&=&
b(X_{0}^*\beta_{0}^*P_{0}-P_{0}^*\beta_{0}X_{0})b^*,
\end{array}\right.$$
Then, we deduce from formula \eqref{COMPACT}:
$$\Hess h_{x}(v)
=
-2
\begin{bmatrix}m&\;\;0\cr 0&\;\;a\cr\end{bmatrix}
B
\block{X_{0}P_{0}^*P_{0}-\beta_{0}^*\beta_{0} X_{0}}{X_{0}^*\beta_{0}^*P_{0}-P_{0}^*\beta_{0} X_{0}}
b^*
=
g\cdot \Hess h_{x_0^q}(v_0).\qedhere
$$
\end{proof}

From the properties of a left action,  
we reduce the search of the eigenvalues and eigenvectors of the Hessian to the case of the notable points.

\begin{corollary}
The Hessian has the same eigenvalues in all points of the critical level~$\Sigma_q$. 
Moreover, the vector $v_0$ is a $\lambda$-eigenvector of $\Hess h_{x_{0}^q}$ if, and only if,
 $g\cdot v_0$ is a $\lambda$-eigenvector of $g\cdot \Hess h_{x_{0}^q}$.
\end{corollary}

\subsection{\bf Eigenvalues, kernel and index at the notable point $x_{0}^q$}
We first have to solve the system
\begin{equation}\label{SYSTEM}
\left\{
\begin{array}{ccl}
-2(X_0P_0^*P_0-\beta_0^*\beta_0 X_0)
&=&
X_0\lambda,\\
-2(X_0^*\beta_0^*P_0 -P_0^*\beta_0 X_0)
&=&
Y_0 \lambda,
\end{array}\right.
\end{equation}
with $T_{0}$ and $P_{0}$ as in \defref{def:mascomplicadoPero}.
We complete $\block{T_0}{P_0}$ to a symplectic matrix 
$B=\begin{bmatrix}\alpha_{0}&T_{0}\cr \beta_{0} &P_{0}\cr\end{bmatrix}
$ with
$\beta_0=\begin{bmatrix}I_p&0\cr 0&0\end{bmatrix}$.
Then $\beta_0^*\beta_0=\begin{bmatrix}I_p&0\cr 0&0\end{bmatrix}$ and
$P_0^*\beta_0=0$.
With these values, the second equation of \eqref{SYSTEM} becomes $Y_0\lambda=0$. For the first one, after a writing of
$X_{0}$ in block of the adequate size,
 $X_0=\begin{bmatrix}a_{1}&a_{2}\cr a_{3}&a_{4}\end{bmatrix}$, this equation is equivalent to the system:
$$2a_{1}=a_{1}\lambda,\quad a_{2}\lambda=0, \quad a_{3}\lambda=0, \quad -2a_{4}=a_{4}\lambda.$$
Therefore, the eigenvalues and eigenvectors are:
\begin{itemize}
\item $\lambda=0$ with $a_{1}=0$, $a_{4}=0$ . So $X_0=\begin{bmatrix}0&a_{2}\cr a_{3}&0\end{bmatrix}$ and $Y_0$ is an arbitrary
skew-Hermitian matrix,
\item $\lambda=2$,  with $Y_{0}=0$ and $a_{2}=0$, $a_{3}=0$, $a_{4}=0$,
\item $\lambda=-2$, with $Y_{0}=0$ and $a_{1}=0$, $a_{2}=0$, $a_{3}=0$.
\end{itemize}
From that, we deduce immediately  the dimension of the kernel and the index. For the kernel, we observe that
$a_{2}\in \HH^{p\times (k-p)}$ and $a_{3}\in \HH^{(n-k-p) \times p}$. 

\begin{corollary}\label{cor:dimker}
The (real) dimension of the kernel of $\Hess h_{x_{0}^q},$ is
$$4p (k-p)+4(n-k-p) p + 3k + 4\frac{k^2-k}{2}=4np-8p^2+2k^2+k.$$
\end{corollary}

\begin{corollary}
The index of the critical point is the dimension of the eigenspace for $\lambda=-2$, that is, the dimension
of the vector space of matrices
$$X_0=\begin{bmatrix}0&0\cr 0&a_{4}\end{bmatrix},$$
which equals $4(n-k-p) (k-p)=4(n-2k+q)q$, with $q=k-p$\end{corollary}

\section{Critical submanifolds}\label{sec:criticalsubmanifolds}

We can  describe  the critical level $\Sigma_q$ as a 
homogeneous space of the group $K_{n,k}=\Sp(n-k)\times \Sp(k)\times \Sp(k)$, which acts transitively.
For  that, we compute the isotropy  subgroup of the notable point $x_0^q=\block{T_0}{P_0}$ 
by solving the equation $g\cdot x_0^q=x_0^q$.
With $g=(m,a,b)\in K_{n,q}$, this  equation becomes the system,
\begin{equation}\label{equa:mabylosotros}
mT_0=T_0b\quad \text{and}\quad aP_0=P_0b.
\end{equation}
We decompose again the matrices in boxes of the adequate sizes:
$$m=\begin{bmatrix}m_{11}&m_{12}\cr m_{21}&m_{22}\end{bmatrix}, \quad
a=\begin{bmatrix}a_{11}&a_{12}\cr a_{21}&a_{22}\end{bmatrix} \quad \text{ and } \quad
b=\begin{bmatrix}b_{11}&b_{12}\cr b_{21}&b_{22}\end{bmatrix}.$$
Replacing these expressions for $m$, $a$ and $b$  in \eqref{equa:mabylosotros}, we obtain
$$m_{11}=b_{11}, \quad m_{21}=0, \quad b_{12}=0, \quad a_{12}=0, \quad a_{22}=b_{22}
\quad \text{and} \quad b_{21}=0.$$
In short, we have
$$m=\begin{bmatrix}m_{11}&0\cr 0&m_{22}\cr\end{bmatrix}\in \Sp(p)\times \Sp(n-k-p),\quad
a=\begin{bmatrix}a_{11}&0\cr 0&a_{22}\cr\end{bmatrix}\in \Sp(p)\times \Sp(k-p),$$
together with
$$b=\begin{bmatrix}m_{11}&0\cr 0&a_{22}\cr\end{bmatrix}.$$

\begin{corollary}\label{ISOTROPY} 
The critical submanifold $\Sigma_q$, with $q=k-p$, is diffeomorphic to the quotient $K_{n,k}/L_{n,k,q}$ with
$$K_{n,k}=\Sp(n-k)\times \Sp(k) \times \Sp(k)$$
and
$$L_{n,k,q}=\Sp(p)\times \Sp(n-k-p)\times \Sp(p)\times \Sp(k-p).$$
The injection $L_{n,k,q}\hookrightarrow  K_{n,k}$ 
is given by
$$(m_{1},m_{2},a_{1},a_{2})\mapsto
\left(
\begin{bmatrix}m_1&0\cr 0&m_2\cr\end{bmatrix},  
\begin{bmatrix}a_1&0\cr 0&a_2\cr\end{bmatrix},
\begin{bmatrix}m_1&0\cr 0&a_2\cr\end{bmatrix}
\right).
$$
\end{corollary}

\begin{corollary}\label{cor:todobien}
The dimension of $\Sigma_{q}$ coincides with the dimension of the kernel of the Hessian, thus 
its is a nondegenerate critical submanifold.
\end{corollary}

\begin{proof}
Using  $\dim \Sp(m)=2m^2+m$, we find 
$\dim \Sigma_{q}=\dim K_{n,k}-\dim L_{n,k,q}=4np -8p^2+2k^2+k$.
The equality $\dim \Sigma_{q}=\dim  \ker \Hess h_{x_{0}^q},$
follows from \corref{cor:dimker}.
\end{proof}

 \begin{proposition}\label{FIBERBUNDLE}
 The critical level $\Sigma_q$ is the total space of a fibre bundle of fibre $\Sp(k)$, of structural group
 $\Sp(k-q)\times \Sp(q)$ and of basis  the product of two Grassmannians $\Gr_{n-k,k-q}\times \Gr_{k,q}$.
\end{proposition}

\begin{proof}
With the notations of Corollary \ref{ISOTROPY}, let $(m,a,b)$ be an element of $K_{n,k}$ and $(m_1,m_2,a_1,a_2)$ an element of $L_{n,k,q}$. 
As proved in the computation of the isotropy group, the action on the right of $L_{n,k,q}$ on $K_{n,k}$ is given by
\begin{equation}\label{ACTION}
(m,a,b)\cdot(m_1,m_2,a_1,a_2) =
\left(m\begin{bmatrix}m_1&0\cr 0&m_2\cr\end{bmatrix}, 
a \begin{bmatrix}a_1&0\cr 0&a_2\cr\end{bmatrix},
b\begin{bmatrix}m_1&0\cr 0&a_2\cr\end{bmatrix}\right).
\end{equation}
We first consider the group of elements $(m_2,a_1)\in \Sp(n-k-p)\times \Sp(p)$.
The quotient of $K_{n,k}$ by this subgroup is 
$$\mathcal{E}=X_{n-k,p}\times X_{k,k-p}\times \Sp(k).$$
Moreover, the action of  $(m_1,a_2)\in\Sp(p)\times \Sp(k-p)$ on 
$(\bar m,\bar a,b)\in\mathcal{E}$ is given by
$$(\bar m,\bar a,b)\cdot (m_1,a_2)=
\left(\bar m m_1,\bar a a_2,b\begin{bmatrix}m_1&0\cr 0&a_2\cr\end{bmatrix}\right).$$
The quotient is known to be the critical submanifold $\Sigma_{q}$.
By definition, this is also the fiber bundle associated to the principal bundle 
$$\Sp(p)\times \Sp(k-p)\to
X_{n-k,p}\times X_{k,k-p}\to
\Gr_{n-k,p}\times \Gr_{k,k-p}$$
and the left $(\Sp(p)\times \Sp(k-p))$-space $\Sp(k)$. 
Its fibre is $\Sp(k)$ and the structural group is $\Sp(p)\times \Sp(k-p)$, see \cite[\S5 of Chapter 4]{MR1249482}
for more details on this construction.
\end{proof}

\begin{example}\label{EXAMPLEINVOLVED}
Let  $n<2k$, and $q=2k-n$ corresponding to the minimal level. The critical manifold
$\Sigma_{2k-n}$ is the total space of a fiber bundle, with fiber $\Sp(k)$, and basis
$$\Gr_{n-k,n-k}\times \Gr_{k,2k-n}= \Gr_{k,n-k}.$$
For instance, if $n=2k-1$, the minimal level corresponds to $q=1$ and $\Sigma_{1}$
fibers over the projective space $\Gr_{k,k-1}=\Gr_{k,1}=\HH P^{k-1}$ with fiber $\Sp(k)$.
\end{example}

\section{Integration of the gradient flow}\label{integration:flow}
In this section we show how to give an explicit description of the gradient flow for the function $h$. There are already  examples in the literature of gradient flows which are naturally integrable: for instance linear functions on compact symmetric spaces \cite{MR3336336}.

According to Proposition \ref{prop:gradient},
the gradient of
 $h(x)= \Tr(P^*P)$ at the point  $x=\block{T}{P}$
is
$$\grad h_x =-2\block{TP^*P}{(PP^*-I)P},$$
so if $\alpha(t)=\block{T(t)}{P(t)}$ the equation
$\alpha^\prime(t)=\grad h_{\alpha(t)}$
reduces to
\begin{equation}\label{MAIN}
T'=-2TP^*P, \quad P'=-2(PP^*-I)P.
\end{equation}
 
Let us fix an initial condition $\alpha(0)=\block{T(0)}{P(0)}$.
According to \cite[Theorem 3.1]{MR3990067}, 
the $k\times k$ matrix $P(0)$ admits a singular value decomposition
\begin{equation}\label{MAT:P}
P(0)=a\,
\begin{bmatrix}I_{p\times p}&0&0\\
0&\diag[c_i]_{q\times q}&0\\
0&0&0_{r\times r}
\end{bmatrix}\,b^*,
\end{equation} 
with $p+q+r=k$, $a,\,b\in\Sp(k)$, $p,q,r\geq 0$ and  $0<c_i<1$.

Analogously,
\begin{equation}\label{MAT:T}
T(0)=\,
m\begin{bmatrix}
0_{p'\times p}&0&0\\
0&\diag[-s_i]_{q\times q}&0\\
0&0&&-I_{r\times r}
\end{bmatrix}\,b^*,
\end{equation} 
with $0<s_i=\sqrt{1-c_i^2}<1$, $m\in \Sp(n-k)$ and $p'+q+r=n-k$. 

Moreover, if $\alpha(0)$ is not a critical point, then $q>0$, according to point 1) of  Proposition \ref{prop:points}.

For each $i=1,\dots,q$ choose constants $C_i>0$ such that
\begin{equation}\label{INIT}
c_i^2=\frac{C_i}{1+C_i}.
\end{equation}

\begin{lemma}\label{SOL:P}
With the notations of \eqref{MAT:P}, the equation $P'=-2(PP^*-I)P$ in \eqref{MAIN} has the solution
$$P(t)=a\,
\begin{bmatrix}I_{p\times p}&0&0\\
0&\diag[c_i(t)]_{q\times q}&0\\
0&0&0_{r\times r}
\end{bmatrix}\,b^*,$$
where
\begin{equation}\label{EXPO}
c_i(t)^2=\frac{C_ie^{4t}}{1+C_ie^{4t}}.
\end{equation}
\end{lemma}

\begin{proof}
We have
$$P'=a\,D^\prime\,b^*$$
while
\begin{align*}
-2(PP^*-I)P=&-2(aDb^*bDa^*-I)aDb^*=\\
&-2a(D^2-I)a^*aDb^*=\\
&-2a(D^2-I)Db^*.
\end{align*}
So we have to prove that
$D^\prime=-2(D^2-I)D$.
This identity is trivially true for the entries $1$ or $0$ in $D$. For the other entries $c_i(t)$, it follows easily from \eqref{EXPO}.

Notice that the initial condition is $P(0)$, because $c_i(0)^2=c_i^2$, as it follows from \eqref{INIT}.

\end{proof}

\begin{lemma}\label{SOL:T}
With the notations of \eqref{MAT:T}, the solution of the equation $T^\prime=-2TP^*P$ in \eqref{MAIN} is
$$T(t)=
m\,
\begin{bmatrix}
0_{p'\times p}&0&0\\
0&\diag[-s_i(t)]_{q\times q}&0\\
0&0&&-I_{r\times r}
\end{bmatrix}\,b^*,
$$
where
\begin{equation}\label{FUNC:S}
s_i(t)^2=1-c_i(t)^2=\frac{1}{1+C_ie^{4t}}.
\end{equation}
\end{lemma}

\begin{proof}
We have
$$-2TP^*P=2m\,\begin{bmatrix}
0_{p'\times p}&0&0\\
0&\diag[s_i(t)c_i(t)^2]_{q\times q}&0\\
0&0&0_{r\times r}
\end{bmatrix}\,b^*,$$
so we only have to show that
$$-s_i^\prime=2s_i(1-s_i^2),$$
which follows easily from \eqref{FUNC:S}.

Notice that the initial condition is $T(0)$ because $s_i(0)=s_i$.
\end{proof}

Finally, notice that the solutions in Lemmas \ref{SOL:P} and \ref{SOL:T} verify
$T^*T+P^*P=I_k$ for all $t$, so we have completely integrated the gradient equation.


\begin{example}On the sphere $S^7=X_{2,1}$, the two coordinates of $x=\block{T}{P}$ are quaternions. The function $h(x)=\vert P\vert^2$ has two critical levels $h=0,1$, each one diffeomorphic to  $S^3$ (they correspond to $P=0$ or $T=0$, respectively).  Take a regular point $\alpha(0)=\block{T_0}{P_0}$.
The flow line passing through it is
$$T(t)=-s(t)\frac{T_0}{\vert T_0\vert}, \quad P(t)=c(t)\frac{P_0}{\vert P_0\vert},$$
where
$$s(t)^2=\frac{1}{1+Ce^{4t}}, \quad c(t)^2=\frac{Ce^{4t}}{1+Ce^{4t}}$$
and
$$C=\frac{\vert P_0\vert}{1-\vert P_0\vert}.$$
Taking limits to $\pm \infty$ we observe that the flow line goes from $\block{T_0/\vert T_0\vert}{0}$ to$\block{0}{P_0/\vert P_0\vert}$.
\end{example}

 
 \def\cprime{$'$}
\providecommand{\bysame}{\leavevmode\hbox to3em{\hrulefill}\thinspace}
\providecommand{\MR}{\relax\ifhmode\unskip\space\fi MR }
\providecommand{\MRhref}[2]{%
  \href{http://www.ams.org/mathscinet-getitem?mr=#1}{#2}
}
\providecommand{\href}[2]{#2}

\end{document}